\documentclass{article}
\usepackage[english]{babel}
\usepackage{amsfonts,amssymb, amsmath, amsthm}
\usepackage[english]{babel}

\newtheorem{theorem}{Theorem}
\newtheorem{proposition}{Proposition}
\newtheorem{lemma}{Lemma}
\theoremstyle{definition}
\newtheorem{remark}{Remark}

\newcommand {\bl} {{\boldsymbol \lambda}}
\newcommand {\bs} {\boldsymbol}
\newcommand {\rk} {\mathop{\rm rank}\nolimits}
\newcommand{\bo}{{\boldsymbol \omega}}
\newcommand{\bd}{{\mathcal B}}
\begin{document}

\title{Bifurcation diagrams and critical subsystems\\
of the Kowalevski gyrostat in two constant fields}

\author{Mikhail P. Kharlamov\\
\it\small Gagarin Street 8, Volgograd, 400131 Russia\\
\it\small e-mail: mharlamov@vags.ru}

\date{}

\maketitle

\begin{abstract}
The Kowalevski gyrostat in two constant fields is known as the
unique example of an~integrable Hamiltonian system with three
degrees of freedom not reducible to a family of systems in fewer
dimensions and still having the clear mechanical interpretation. The
practical explicit integration of this system can hardly be obtained
by the existing techniques. Then the challenging problem becomes to
fulfil the qualitative investigation based on the study of the
Liouville foliation of the phase space. As the first approach to
topological analysis of this system we find the stratified critical
set of the momentum map; this set consists of the trajectories with
number of frequencies less than three. We obtain the equations of
the bifurcation diagram in three-dimensional space. These equations
have the form convenient for the classification of the bifurcation
sets induced on 5-dimensional iso-energetic levels.
\end{abstract}




%\tableofcontents

%\large

\section{Introduction}\label{sec1}

The famous integrable case of S.\,Kowalevski of the motion of a
heavy rigid body about a fixed point \cite{Kowa} has received
several generalizations. Some of them suppose restrictions to
submanifolds in the phase space (partial cases), others are far from
mechanics involving potential functions on the configuration space
$SO(3)$ with singularities. The most essential generalization having
the clear mechanical sense was found by A.G.\,Reyman and
M.A.\,Semenov-Tian-Shansky in the work \cite{ReySem}. The authors
introduce the dynamical system on the dual space of the Lie algebra
$e(p,q)$ of the Lie group defined as the semi-direct product of
$SO(p)$ and $q$ copies of ${\mathbb{R}}^p$. Such systems are known
as the Euler equations on Lie (co)algebras \cite{Bogo}. The case
$p=3,q=2$ corresponds to the Euler--Poisson equations of the motion
of a gyrostat in two constant fields.

For a rigid body without gyrostatic momentum, the model of two
constant fields was introduced by O.I.\,Bogoyavlensky \cite{Bogo}.
The physical object can be either a heavy electrically charged rigid
body rotating in gravitational and constant electric fields, or a
heavy magnet rotating in gravitational and constant magnetic fields.
The corresponding equations are Hamiltonian on the orbit of
coadjoint action on $e(3,2)^*$ of the Lie group defined as the
semi-direct product
${SO(3)\times({\mathbb{R}}^3\otimes{\mathbb{R}}^3)}$. The typical
orbit is diffeomorphic to $TSO(3)\cong {\mathbb{R}}^3\times SO(3)$.
Therefore, the gyrostat in two constant fields is the Hamiltonian
system with three degrees of freedom. Bogoyavlensky \cite{Bogo}
suggested the conditions of the Kowalevski type and found the
analogue of the Kowalevski integral $K$ for the top in two constant
fields. H.\,Yehia \cite{Yeh} generalized this integral for the
Kowalevski gyrostat in two constant fields. Almost simultaneously
with Yehia, I.V.\,Komarov \cite{Komar} and L.N.\,Gavrilov
\cite{Gavril} proved the Liouville integrability of the Kowalevski
gyrostat in the gravity field. But for two constant fields the
Kowalevski gyrostat was not considered integrable due to the fact
that the existence of the second field destroys the axial symmetry
of the potential and, consequently, the corresponding cyclic
integral. Finally, Reyman and Semenov-Tian-Shansky \cite{ReySem}
found the Lax representation with a spectral parameter for the
family of Euler equations on $e(p,q)^*$. For $e(3,2)$ this
representation immediately gave rise to the new integral for the
Kowalevski gyrostat in two constant fields. For the classical
Kowalevski top this integral turns into the square of the cyclic
integral.

The Kowalevski gyrostat in two constant fields does not have any
explicit symmetry groups and, therefore, is not reducible, in a
standard way, to a family of systems with two degrees of freedom.
Phase topology of such systems has not been studied yet. The theory
of $n$-dimensional integrable systems started in \cite{FomSG} is not
illustrated by an application to any real irreducible physical or
geometrical problem with $n > 2$.

In the paper \cite{BobReySem}, the authors give a detailed
exposition of the results of \cite{ReySem} as well as a study of the
algebraic geometry of the Lax pair for the generalized Kowalevski
system. They announce the possibility of its integration by the
finite-band techniques and fulfil such integration for the classical
top. For two constant fields the integration of the Kowalevski top
is not given up-to-date. The problem of the Kowalevski {\it
gyrostat} motion in two constant fields is not studied at all. The
technical difficulties here are extremely high. It is not likely
that, in the general regular case, the analytical solutions can be
obtained having the form useful for the qualitative topological
analysis or the computer simulation. However, there is a good
experience of studying the critical subsystems, i.e., the systems
with two degrees of freedom induced on 4-dimensional invariant
submanifolds of the phase space. For the Kowalevski top in two
constant fields we have now the complete description of all
singularities of the momentum map \cite{ZotRCD}, \cite{Kh32},
\cite{Kh34}, \cite{KhRCD1}, \cite{KhSav}, \cite{KhRCD2} and the
classification of the bifurcation diagrams for the restriction of
this map to \mbox{5-dimensional} iso-energetic surfaces
\cite{KhZot}, \cite{Kh361}, \cite{Kh362}. This result is a necessary
and highly complicated part of the study of Liouville foliation of
the integrable system and shows the actual need in the
generalization of the Liouville invariants theory \cite{BolsFom} for
the dimensions greater than two.

The present paper contains similar results for the Kowalevski
gyrostat in two constant fields. The 6-dimensional phase space is
stratified by the rank of the momentum map. We find the equations of
invariant submanifolds on which the induced systems are Hamiltonian
with less than three degrees of freedom (critical manifolds of rank
0,1, or 2). We straightforwardly prove that the image of these
critical manifolds (the bifurcation diagram) lies in the
discriminant set of the algebraic curve of the Lax representation
given in \cite{BobReySem}. Moreover, the spectral parameter on the
Lax curve is explicitly expressed in terms of the constant $s$ of
the additional partial integral arising on the critical
submanifolds. It then follows that the equations of the surfaces
containing the bifurcation diagram are written in the parametric
form such that the parameters are the energy constant $h$ and the
constant $s$ of the partial integral. Fixing the value of $h$ we
come to explicit equations of the bifurcation diagrams induced on
iso-energetic levels. The problem of classification of these
diagrams seems quite complicated due to the existence of several
physical parameters. Nevertheless, it is certainly solvable with the
help of contemporary computer programs of analytical calculations.

First we show that the number of physical parameters for the
gyrostat in two constant fields can be reduced by a simple
procedure, which may be called the orthogonalization of the fields.
More precisely, for the problems of gyrostat motion there exists a
group of diffeomorphisms of the phase spaces (mentioned above orbits
of the coadjoint action) that is an equivalence group for the
corresponding dynamical systems. It appears that each equivalence
class contains a problem with an orthonormal pair of radius vectors
of the centers of forces application and with an orthogonal pair of
the intensity vectors. Such force field is characterized by only one
essential parameter---the ratio of the modules of the intensity
vectors. For a dynamically symmetric gyrostat having the centers of
forces application in the equatorial plane, the orthogonalization
procedure along with the appropriate choice of the measure units
leave, in addition to the forces ratio, only two physical parameters
of the body itself, namely, the ratio of the equatorial and axial
inertia moments and the non-zero axial component of the gyrostatic
momentum. In the generalized Kowalevski case the first of them
equals 2. Thus, the whole problem has, in fact, two essential
parameters. In particular, each of the critical four-dimensional
submanifolds found below provides a two-parametric family of
completely integrable Hamiltonian systems with two degrees of
freedom.

\section{Gyrostat equations and parametrical reduction}\label{sec2}

Consider a rigid body $\bd$ rotating around a fixed point $O$.
Choose a trihedral at $O$ moving along with the body and refer to it
all vector and tensor objects. Denote by ${\bf e}_1{\bf e}_2{\bf
e}_3$ the canonical unit basis in ${\mathbb{R}}^3$; then the moving
trihedral itself is represented as $O{\bf e}_1{\bf e}_2{\bf e}_3$.
Let $\bo$ be the vector of the angular velocity of $\bd$. Suppose
that $\bd$ is bearing an axially symmetric rigid rotor $\bd'$
rotating freely around its symmetry axis fixed in $\bd$. Such system
of two bodies is the simplest model of a gyrostat. The notion of a
gyrostat was introduced by N.E.\,Zhukovsky \cite{Zhuk} for a body
having cavities totally filled with homogeneous fluid. Both models
have the common feature usually taken as the definition of a
gyrostat: the total angular momentum of such system is ${\bf M} =
{\bf I} \bo +\bl$, where the inertia tensor ${\bf I}$ and the vector
$\bl$ (called the gyrostatic momentum) are constant with respect to
the moving trihedral. Using the term "gyrostat"\, we always suppose
$\bl \ne 0$. In the case $\bl =0$ we use the terms "rigid body"\, or
"top"\, instead. The top is usually supposed to have a dynamical
symmetry axis.

Let ${\bf M}_F$ denote the moment of external forces with respect to
$O$ (the rotating moment). Constant field is a force field inducing
the rotating moment of the form ${\bf r}\times{\boldsymbol \alpha}$
with constant vector ${\bf r}$ and with ${\boldsymbol \alpha}$
corresponding to some physical vector fixed in inertial space; ${\bf
r}$ points from $O$ to the center of application of the field,
${\boldsymbol \alpha}$ is the field intensity.

For {\it two constant fields} the rotating moment is ${\bf M}_F={\bf
r}_1\times{\boldsymbol \alpha}+{\bf r}_2\times{\boldsymbol \beta}$
with ${\bf r}_1,{\bf r}_2$ constant in the body and ${\boldsymbol
\alpha},{\boldsymbol \beta}$ corresponding to the vectors fixed in
inertial space. Obviously, ${\bf M}_F$ can be represented as the
moment of {\it one constant field} if either ${\bf r}_1\times{\bf
r}_2=0$ or ${\boldsymbol \alpha}\times{\boldsymbol \beta}=0$.
Suppose that
\begin{equation}\label{eq2_1}
{\bf r}_1\times{\bf r}_2 \ne 0, \quad {\boldsymbol
\alpha}\times{\boldsymbol \beta} \ne 0.
\end{equation}
Two constant fields satisfying (\ref{eq2_1}) are said to be {\it
independent}.

The equations defining the respective evolution of ${\bf M},{\bs
\alpha},{\bs \beta}$ in two constant fields are
\begin{equation}\label{eq2_2}
\begin{array}{l}
\displaystyle{\frac{d{\bf M}}{dt} = {\bf M}\times \bo + {\bf
r}_1\times{\boldsymbol \alpha}+{\bf r}_2\times{\boldsymbol
\beta}},\quad  \displaystyle{\frac{d{\bs \alpha}}{dt} = {\bs \alpha}
\times \bo}, \quad \displaystyle{\frac{d{\bs \beta}}{dt} = {\bs
\beta} \times \bo}.
\end{array}
\end{equation}
These equations are Euler equations in the space
${\mathbb{R}}^9({\bf M},{\bs \alpha},{\bs \beta})$ considered as the
dual space to the semi-direct sum $so(3)+({\mathbb{R}}^3 \otimes
{\mathbb{R}}^3)$. The Lie--Poisson bracket applied to the coordinate
functions yields
\begin{equation}\label{eq2_3}
\begin{array}{c}
\{M_i,M_j\}=\varepsilon_{i j k} M_k, \quad
\{M_i,\alpha_j\}=\varepsilon_{i j k} \alpha_k, \quad
\{M_i,\beta_j\}=\varepsilon_{i j k} \beta_k, \\[1.5mm]
\{\alpha_i,\alpha_j\}=0, \quad \{\alpha_i,\beta_j\}=0, \quad
\{\beta_i,\beta_j\}=0.
\end{array}
\end{equation}
Such bracket is non-degenerate on each orbit of the coadjoint
action. The orbits are defined by the geometric integrals (common
level of the Casimir functions)
\begin{equation}\label{eq2_4}\notag
\begin{array}{l}
{\boldsymbol \alpha} \cdot {\boldsymbol \alpha}=c_{11}, \quad
{\boldsymbol \beta} \cdot {\boldsymbol \beta}=c_{22}, \quad
{\boldsymbol \alpha} \cdot {\boldsymbol \beta}=c_{12}.
\end{array}
\end{equation}
If $c_{11}>0, c_{22}>0, c_{12}^2 < c_{11} c_{22}$, then the orbit in
${\mathbb{R}}^9$ is diffeomorphic to ${\mathbb{R}}^3 \times SO(3)$,
and the induced Hamiltonian system has three degrees of freedom (see
\cite{Bogo}, \cite{BobReySem} for the details). From physical point
of view the constants $c_{11}, c_{22}, c_{12}$ characterize the
force fields intensities. Along with the coordinates of ${\bf
r}_1,{\bf r}_2$ in the moving frame, we have 9 parameters of the
interaction of the body with the external forces. We now show how to
reduce this number.

Introduce some notation.

Let $L(n,k)$ be the space of $n \times k$-matrices. Put
$L(k)=L(k,k)$.

Identify ${\mathbb{R}}^6={\mathbb{R}}^3 \times {\mathbb{R}}^3$ with
$L(3,2)$ by the isomorphism $j$ that joins two columns
$$
A=j({\bf a}_1 ,{\bf a}_2 ) = \left\| {{\bf a}_1 \;{\bf a}_2 }
\right\| \in L(3,2), \qquad {\bf a}_1 ,{\bf a}_2 \in {\mathbb{R}}^3.
$$

For the inverse map, we write
$$
j^{ - 1} (A) = ({\bf c}_1 (A),{\bf c}_2 (A)) \in {\mathbb{R}}^3
\times {\mathbb{R}}^3, \qquad A \in L(3,2).
$$

If $A,B \in L(3,2)$, ${\bf a} \in {\mathbb{R}}^3 $, by definition,
put
\begin{equation} \label{eq2_5}
\begin{array}{l}
A \times B = \sum\limits_{i = 1}^2 {{\bf c}_i (A) \times {\bf c}_i
(B) \in {\mathbb{R}}^3 ;}\\
{\bf a} \times A = j({\bf a} \times {\bf c}_1 (A),{\bf a} \times
{\bf c}_2 (A)) \in L(3,2).
\end{array}
\end{equation}

\begin{lemma}\label{lem1} Let $\Lambda  \in SO(3)$, $D \in GL(2,{\mathbb{R}})$, ${\bf
a} \in {\mathbb{R}}^3 $, $A,B \in L(3,2)$. Then
\begin{equation}\label{eq2_6}\notag
\begin{array}{ll}
\Lambda (A \times B) = (\Lambda A) \times (\Lambda B); &
(AD^{-1})\times (BD^T) = A \times B; \\
\Lambda ({\bf a} \times A) = (\Lambda {\bf a}) \times (\Lambda A); &
{\bf a} \times (AD) = ({\bf a} \times A)D.
\end{array}
\end{equation}
\end{lemma}

The proof is by direct calculation.

In notation (\ref{eq2_5}) we write Eqs.\,(\ref{eq2_2}) in the form
\begin{equation}\label{eq2_7}
\displaystyle{{\bf I} \frac {d{\boldsymbol \omega }}{dt} = ({\bf
I}{\boldsymbol \omega }+\bl) \times {\boldsymbol \omega } + A \times
U,} \qquad \displaystyle{\frac {dU}{dt} =  - {\boldsymbol \omega }
\times U.}
\end{equation}
Here $A=j({\bf r}_1,{\bf r}_2)$ is a constant matrix,
$U=j({\boldsymbol \alpha},{\boldsymbol \beta})$. The phase space of
(\ref{eq2_7}) is $\{({\boldsymbol \omega},U)\}={\mathbb{R}}^3\times
L(3,2)$.

In fact, $U$ in (\ref{eq2_7}) is restricted by the geometric
integrals; i.e., for some constant symmetric matrix $C \in L(2)$
\begin{equation}\label{eq2_8}
U^T U = C.
\end{equation}
Let ${\mathcal O}$ be the set defined by Eq.\,(\ref{eq2_8}) in
$L(3,2)$. In order to emphasize the $C$-dependence, we write
${\mathcal O}={\mathcal O}(C)$.

Let $\mathfrak{P}=({\bf I},\bl, A,C)$ denote the complete set of
constant parameters of the problem. Denote by $X_\mathfrak{P}$ the
vector field on ${\mathbb{R}}^3 \times {\mathcal O}(C)$ induced by
(\ref{eq2_7}). Given the set $\mathfrak{P}$, the problem of motion
of the gyrostat in two constant fields described by the dynamical
system $X_\mathfrak{P}$ will be called, for short, the {\it
DG-problem}.

Associate to $\Lambda \in SO(3)$, $D \in GL(2,{\mathbb{R}})$ the
linear automorphisms $\Psi(\Lambda,D)$ and $\psi(\Lambda,D)$ of
${\mathbb{R}}^3 \times L(3,2)$ and $L(3)\times {\mathbb{R}}^3\times
L(3,2) \times L(2)$
\begin{equation}\label{eq2_9}
\begin{array}{l}
\Psi(\Lambda,D)({\boldsymbol \omega},U)=(\Lambda {\boldsymbol
\omega},\Lambda U
D^T), \\
\psi(\Lambda,D)({\bf I},\bl,A,C)=(\Lambda {\bf I} \Lambda^T, \Lambda
\bl, \Lambda A D^{-1}, D C D^T).
\end{array}
\end{equation}

Eqs.\,(\ref{eq2_8}) and (\ref{eq2_9}) imply
$\Psi(\Lambda,D)({\mathbb{R}}^3 \times {\mathcal
O}(C))={\mathbb{R}}^3 \times {\mathcal O}(D C D^T)$. Using
Lemma~\ref{lem1} we obtain the following statement.

\begin{lemma}\label{lem2}
For each $(\Lambda,D) \in SO(3)\times GL(2,{\mathbb{R}})$, we have
$$
\Psi(\Lambda,D)_*(X_\mathfrak{P}(v))=X_{\psi(\Lambda,D)(\mathfrak{P})}(\Psi(\Lambda,D)(v)),
\qquad v \in {\mathbb{R}}^3 \times {\mathcal O}(C).
$$
\end{lemma}

Thus, any two DG-problems determined by the sets of parameters
$\mathfrak{P}$ and $\psi(\Lambda,D)(\mathfrak{P})$ are completely
equivalent.

Let us call a DG-problem {\it canonical} if the centers of
application of forces lie on the first two axes of the moving
trihedral at unit distance from $O$ and the intensities of the
forces are orthogonal to each other.

\begin{theorem} \label{th1}
For each DG-problem with independent forces there exists an
equivalent canonical problem. Moreover, in both equivalent problems
the centers of application of forces belong to the same  plane in
the body containing the fixed point.
\end{theorem}
\begin{proof} Let the DG-problem with the set of
parameters $ \mathfrak{P}=({\bf I},\bl,A,C)$ satisfy~(\ref{eq2_1}).
This means that the symmetric matrices $A_*=(A^T A)^{-1}$ and $C$
are positively definite. According to the well-known fact from
linear algebra, $A_*$ and $C$ can be reduced, respectively, to the
identity matrix and to a diagonal matrix via the same conjugation
operator
\begin{equation}\label{eq2_10}\notag
DA_*D^T=E,\quad DCD^T= \mathop{\rm diag}\nolimits \{a^2,b^2\},\qquad
D \in GL(2,{\mathbb{R}}),\;a,b \in {\mathbb{R}}_+.
\end{equation}
Then ${\bf c}_1(AD^{-1})$ and ${\bf c}_2(AD^{-1})$ form an
orthonormal pair in ${\mathbb{R}}^3$. There exists $\Lambda \in
SO(3)$ such that $\Lambda{\bf c}_i(AD^{-1})={\bf e}_i \,(i=1,2)$.
The first statement is obtained by applying Lemma \ref{lem2} with
the previously chosen $\Lambda,D$ to the initial vector field
$X_\mathfrak{P}$.

To finish the proof, notice that the transformation $A \mapsto
AD^{-1}$ preserves the span of ${\bf c}_1(A)$, ${\bf c}_2(A)$. The
matrix $\Lambda$ in (\ref{eq2_9}) stands for the change of the
moving trihedral. Therefore, if ${\bf a}\in {\mathbb{R}}^3$
represents some physical vector in the initial problem, then
$\Lambda{\bf a}$ is the same vector with respect to the body in the
equivalent problem.
\end{proof}

\begin{remark}\label{rem1} The fact that any DG-problem can be reduced to the
problem with {\it one} of the pairs ${\bf r}_1,{\bf r}_2$ or
${\boldsymbol \alpha}, {\boldsymbol \beta}$ orthonormal is obvious.
Simultaneous orthogonalization of {\it both} pairs was first
established in {\rm \cite{Kh34}} for a rigid body and crucially
simplifies all calculations.
\end{remark}

It follows from Theorem~\ref{th1} that, without loss of generality,
for independent forces we may suppose
\begin{gather}
{\bf r}_1={\bf e}_1,\quad {\bf r}_2={\bf e}_2, \label{eq2_11} \\
{\boldsymbol \alpha} \cdot {\boldsymbol \alpha}=a^2, \;{\boldsymbol
\beta} \cdot {\boldsymbol \beta}=b^2, \;{\boldsymbol \alpha} \cdot
{\boldsymbol \beta}=0. \label{eq2_12}
\end{gather}

Change, if necessary, the order of ${\bf e}_1, {\bf e}_2$ (with
simultaneous change of the direction of ${\bf e}_3$) to obtain $a
\geqslant b > 0$.

Consider a dynamically symmetric top in two constant fields with the
centers of application of forces in the equatorial plane of its
inertia ellipsoid. Choose a moving trihedral such that $O{\bf e}_3$
is the symmetry axis. Then the inertia tensor ${\bf I}$ becomes
diagonal. Let $a=b$. For any $\Theta \in SO(2)$ denote by ${\hat
\Theta} \in SO(3)$ the corresponding rotation of ${\mathbb{R}}^3$
about $O{\bf e}_3$. Take in (\ref{eq2_9}) $\Lambda={\hat \Theta}$,
$D=\Theta$. Under the conditions (\ref{eq2_11}), (\ref{eq2_12}),
$\psi= \textrm{Id}$ and $\Psi$ becomes the symmetry group. The
system (\ref{eq2_7}) has the cyclic integral $ {\bf I}{\boldsymbol
\omega}\cdot (a^2{\bf e}_3 - {\boldsymbol \alpha}\times {\boldsymbol
\beta})$. Therefore it is possible to reduce such a DG-problem to a
family of systems with two degrees of freedom. For the analogue of
the Kowalevski case this system becomes integrable \cite{Yeh}.

Let us call a DG-problem {\it irreducible} if, in its canonical
representation,
\begin{equation}\label{eq2_13}
a>b>0.
\end{equation}

The following statements are needed in the future; they also reveal
some features of a wide class of DG-problems.
\begin{lemma} \label{lem3}
In an irreducible DG-problem, the body has exactly four equilibria.
\end{lemma}

\begin{proof}
The set of singular points of (\ref{eq2_7}) is defined by $
{\boldsymbol \omega}=0,\;A \times U=0$. For the equivalent canonical
problem with (\ref{eq2_11}) we have
\begin{equation}\label{eq2_14}\notag
{\bf e}_1 \times {\boldsymbol \alpha}+{\bf e}_2 \times {\boldsymbol
\beta}=0.
\end{equation}
Then the four vectors ${\bf e}_1, {\boldsymbol \alpha}, {\bf e}_2,
{\boldsymbol \beta}$ are parallel to the same plane and $\left| {\bf
e}_1 \times {\boldsymbol \alpha} \right|=\left| {\bf e}_2 \times
{\boldsymbol \beta} \right|$. Given (\ref{eq2_13}), this equality
yields
\begin{equation}\label{eq2_15}
{\boldsymbol \alpha}=\pm a{\bf e}_1, \; {\boldsymbol \beta}=\pm b
{\bf e}_2.
\end{equation}
Thus, in the canonical irreducible system, an equilibrium takes
place only if the radius vectors of the centers of application are
parallel to the corresponding fields intensities.
\end{proof}

Note that the existence of the gyrostatic momentum does not change
the equilibria. Therefore, the result here is the same as in the
case of a rigid body in two constant fields \cite{KhZot}.

\begin{lemma} \label{lem4}
Let an irreducible DG-problem in its canonical form have the
diagonal inertia tensor ${\bf I}=\mathop{\rm diag}\nolimits
\{I_1,I_2,I_3\}$ and $\bl =0$. Then the body has the following
families of periodic motions of pendulum type
\begin{eqnarray}
&& P_1: \left\{ \begin{array}{c} {\boldsymbol \omega } = \varphi ^
{\boldsymbol \cdot} {\bf e}_1 , \quad {\boldsymbol \alpha } \equiv
\pm a{\bf e}_1,\quad {\boldsymbol \beta } = b({\bf e}_2 \cos
\varphi - {\bf e}_3 \sin \varphi ), \\
I_1 \varphi ^{ {\boldsymbol \cdot}  {\boldsymbol \cdot} }  =  -
b\sin \varphi ;
\end{array} \right. \label{eq2_16} \\
&& P_2: \left\{ \begin{array}{c} {\boldsymbol \omega } = \varphi ^
{\boldsymbol \cdot}  {\bf e}_2 , \quad {\boldsymbol \beta } \equiv
\pm b{\bf e}_2 , \quad {\boldsymbol \alpha } = a({\bf
e}_1 \cos \varphi + {\bf e}_3 \sin \varphi ), \\
I_2 \varphi ^{ {\boldsymbol \cdot}  {\boldsymbol \cdot} }  =  -
a\sin \varphi;
\end{array} \right. \label{eq2_17} \\
&& P_3: \left\{ \begin{array}{c} {\boldsymbol \omega } = \varphi ^
{\boldsymbol \cdot}  {\bf e}_3 , \quad {\boldsymbol
\alpha}\times{\boldsymbol \beta} \equiv  \pm a b{\bf e}_3 ,
\\
{\boldsymbol \alpha } = a({\bf e}_1 \cos \varphi  - {\bf e}_2 \sin
\varphi ),\quad {\boldsymbol \beta } =  \pm b({\bf e}_1 \sin \varphi
+ {\bf e}_2
\cos \varphi ), \\
I_3 \varphi ^{ {\boldsymbol \cdot}  {\boldsymbol \cdot} }  =  - (a
\pm b)\sin \varphi.
\end{array} \right. \label{eq2_18}
\end{eqnarray}
If $\bl \ne 0$ but $\bl = \lambda {\bf e}_i$ for some $i=1,2,3$,
then the only family remained is $P_i$ with the corresponding index.
\end{lemma}

The proof is obvious. The families (\ref{eq2_16})--(\ref{eq2_18})
were first found in \cite{Kh34} (the case $\bl=0$). Note that for
two constant fields these families are the only motions with a fixed
direction of the angular velocity. In particular, the body in two
independent constant fields does not have any uniform rotations.

\section{Critical set of the Kowalevski gyrostat}\label{sec3}
Suppose that the irreducible DG-problem has the diagonal inertia
tensor with the principal moments of inertia satisfying the ratio
2:2:1, the gyrostatic momentum is directed along the dynamical
symmetry axis $\bl = \lambda {\bf e}_3$ and the centers of the
fields application lie in the equatorial plane ${\bf r}_1 \bot {\bf
e}_3, {\bf r}_2 \bot {\bf e}_3$. These are the conditions of the
integrable case \cite{ReySem} of the Kowalevski gyrostat in two
constant fields. The orthogonalization procedure in this case does
not change the ${\bf e}_3$-axis and we obtain (\ref{eq2_11}),
(\ref{eq2_12}). Choosing the appropriate units of measurement,
represent Eqs.\,(\ref{eq2_7}) in the form
\begin{equation}\label{eq3_1}
\begin{array}{c}
2{\dot \omega}_1   = \omega _2 (\omega _3-\lambda) + \beta _3 ,\;
2{\dot \omega}_2   =  - \omega _1 (\omega _3-\lambda) - \alpha _3
,\;
{\dot \omega}_3  = \alpha _2  - \beta _1 , \\[1.5mm]
{\dot \alpha}_1   = \alpha _2 \omega _3  - \alpha_3 \omega_2,\qquad
{\dot \beta}_1   = \beta _2 \omega _3 -
\beta_3 \omega_2, \\[1.5mm]
{\dot \alpha}_2   = \alpha _3 \omega _1  - \alpha_1 \omega_3,\qquad
{\dot \beta}_2   = \beta _3 \omega _1 -
\beta_1 \omega_3, \\[1.5mm]
{\dot \alpha}_3   = \alpha _1 \omega _2  - \alpha_2 \omega_1,\qquad
{\dot \beta}_3  = \beta _1 \omega _2 - \beta_2 \omega_1.
\end{array}
\end{equation}
The phase space is $P^6={\mathbb{R}}^3 \times \mathcal{O}$, where
$\mathcal{O} \subset {\mathbb{R}}^3 \times {\mathbb{R}}^3$ is
defined by (\ref{eq2_12}); $\mathcal{O}$ is diffeomorphic to
$SO(3)$.

The complete set of the first integrals in involution on $P^6 $
includes the energy integral~$H$, generalized Kowalevski
integral~$K$~\cite{Bogo}, \cite{Yeh}, and  the integral~$G$ found
in~\cite{ReySem}. After the parametrical reduction, these integrals
are
\begin{equation}\label{eq3_2}\notag
\begin{array}{l}
\displaystyle{H = \omega _1^2  + \omega _2^2  + \frac{1} {2}
\omega_3^2 -
\alpha _1  - \beta _2,} \\
K = (\omega _1^2  - \omega _2^2  + \alpha _1  - \beta _2 )^2 +
(2\omega _1 \omega _2  + \alpha _2  + \beta _1 )^2 + \\[2mm]
\phantom{K = (}+ 2\lambda[(\omega_3-\lambda) ( \omega_1^2+ \omega_2
^ 2)
+ 2 \omega_1 \alpha_3 + 2 \omega_2 \beta_3 ] , \\
\displaystyle{G =\frac{1} {4} (M_{\alpha}^2+ M_{\beta}^2) + \frac{1}
{2} (\omega_3 - \lambda) M_{\gamma} - b^2 \alpha _1  - a^2 \beta
_2.}
\end{array}
\end{equation}
Here $M_{\alpha}=({\bf I}{\bs \omega}+{\bs \lambda}){\bs \cdot} {\bs
\alpha}$, $M_{\beta}=({\bf I}{\bs \omega}+{\bs \lambda}){\bs \cdot}
{\bs \beta}$, $M_{\gamma}=({\bf I}{\bs \omega}+{\bs \lambda}){\bs
\cdot} ({\bs \alpha} \times {\bs \beta})$.

Introduce the momentum map
\begin{equation}\label{eq3_4}
J=G\times K\times H: P^6 \rightarrow {\mathbb{R}}^3
\end{equation}
and denote by ${\mathfrak C} \subset P^6$ the set of critical points
of $J$. By definition, the bifurcation diagram of $J$ is the set
$\Sigma \subset {\mathbb{R}}^3$ over which $J$ fails to be locally
trivial; $\Sigma$ defines the cases when the integral manifolds
\begin{equation}\label{eq3_5}\notag
J_c=J^{-1}(c),   \quad c=(g,k,h) \in {\mathbb{R}}^3
\end{equation}
change its topological (and smooth) type. To find ${\mathfrak C} $
and $\Sigma$ is the necessary part of the global topological
analysis of the problem.

It follows from Liouville--Arnold theorem that for $c \notin \Sigma$
the manifold $J_c$, if not empty, is the union of three-dimensional
tori. The considered Hamiltonian system on $P^6$ is non-degenerate
at least for small enough values of $b$. Therefore the trajectories
on such tori are almost everywhere quasi-periodic with three
independent frequencies. The critical set ${\mathfrak C}$ is
preserved by the phase flow and consists of the trajectories having
less than three frequencies. We call these trajectories {\it the
critical motions}. The set ${\mathfrak C}$ is stratified by the rank
of $J$. Let ${\mathfrak C}_j=\{\zeta \in {\mathfrak C}: \mathop{\rm
rank}\nolimits J(\zeta)=j\}$ ($j=0,1,2$). It is natural to expect
that ${\mathfrak C}_j$ consists of the Liouville tori of
dimension~$j$ and the image $J({\mathfrak C}_j)$, as a subset of
$\Sigma$, is a smooth surface $\Sigma_j$ of dimension~$j$. More
precisely, for each $j \leqslant 2$ we have to take
$$\Sigma_j=
J({\mathfrak C}_j)\backslash\bigcup_{i=0}^{j-1} J({\mathfrak C}_i).
$$
Then, as a whole, we may consider $\Sigma$ as a two-dimensional
cell complex, $\Sigma_j$ as its $j$-skeleton. For $j=1,2$ we will
have $\partial \Sigma_j \subset \Sigma_{j-1}$.

For $c \in \Sigma_2$ the set $J_c \cap {\mathfrak C}$ consists of
two-dimensional tori. Take the union of such tori over the values
$c$ from some open subset in $\Sigma_2$. The dynamical system
induced on this union will be Hamiltonian with two degrees of
freedom. Vice versa, let $M$ be a submanifold in $P^6$, $\dim M =
4$, and suppose that the induced system on $M$ is Hamiltonian. Then
obviously $M \subset {\mathfrak C}$. This speculation gives a useful
tool to find out whether a common level of functions consists of
critical points of $J$.

\begin{lemma}\label{lem5}
Consider a system of equations
\begin{equation}\label{eq3_6}
f_1=0, \dots, f_{2k}=0
\end{equation}
on a domain $W$ open in the phase space $P^{2n}$ of the integrable
Hamiltonian system $X$. Let $M \subset W$ be the set defined by
$(\ref{eq3_6})$. Suppose

$(i)$ $f_1,\dots,f_{2k}$ are smooth functions independent on $M$;

$(ii)$ $Xf_1=0,\dots, Xf_{2k}=0$ on $M$;

$(iii)$ the matrix of the Poisson brackets $\|\{f_i,f_j\}\|$ is
non-degenerate almost everywhere on $M$.

Then $M$ consists of critical points of the momentum map.
\end{lemma}

\begin{proof}
Conditions (i), (ii) imply that $M$ is a smooth
$(2n-2k)$-dimensional manifold invariant under the restriction of
the phase flow to the open set $W$. Condition (iii) means that the
closed \mbox{2-form} induced on $M$ by the symplectic structure on
$P^{2n}$ is almost everywhere non-degenerate. Thus the flow on $M$
is almost everywhere Hamiltonian with $n-k$ degrees of freedom. It
inherits the property of complete integrability. Then almost all its
integral manifolds consist of $(n-k)$-dimensional tori and therefore
lie in the critical set of the momentum map. Since $M$ is closed in
$W$ and the critical set is closed in $P^{2n}$, we conclude that $M$
totally consists of the critical points of the momentum map.
\end{proof}

\begin{remark}\label{rem2}
In our case $n=3$ and the above lemma is applied in the situations
when $k=1$ or $k=2$. The critical set and the bifurcation diagram of
the map $(\ref{eq3_4})$ in the case $\lambda=0$ are known. The
critical set is described by one system of the type~$(\ref{eq3_6})$
with $k=2$ and three systems of the type~$(\ref{eq3_6})$ with $k=1$.
The complete presentation of these results and the list of
publications are given in~{\rm \cite{KhRCD1}}, {\rm \cite{Kh362}}.
Except for the partial integrable case of Bogoyavlensky~{\rm
\cite{Bogo}} (case $K=0$), all of the critical subsystems have been
either explicitly integrated or reduced to separated systems of
equations {\rm \cite{KhSav}}, {\rm \cite{KhRCD2}}, {\rm
\cite{Kh361}}.
\end{remark}

Introduce the change of variables \cite{Kh32} based on the change
given by S.\,Ko\-walev\-ski and on the Lax representation
\cite{ReySem} ($i^2=-1$)
\begin{equation}\label{eq3_7}
\begin{array}{c}
x_1 = (\alpha_1  - \beta_2) + i(\alpha_2  + \beta_1),\quad
x_2 = (\alpha_1  - \beta_2) - i(\alpha_2  + \beta_1 ), \\
y_1 = (\alpha_1  + \beta_2) + i(\alpha_2  - \beta_1), \quad y_2 =
(\alpha_1  + \beta_2) -
i(\alpha_2  - \beta_1), \\
z_1 = \alpha_3  + i\beta_3, \quad
z_2 = \alpha_3  - i\beta_3,\\
w_1 = \omega_1  + i\omega_2 , \quad w_2 = \omega_1  - i\omega_2,
\quad w_3 = \omega_3.
\end{array}
\end{equation}
Then Eqs.\,(\ref{eq3_1}) yield
\begin{equation}\label{eq3_8}
\begin{array}{c}
2w'_1  =  - w_1 (w_3-\lambda) - z_1 ,\quad 2w'_2  = w_2
(w_3-\lambda) + z_2, \quad
2w'_3 = y_2  - y_1, \\
\begin{array}{ll}
{x'_1  =  - x_1 w_3  + z_1 w_1,} & {x'_2  = x_2 w_3  - z_2 w_2,}
\\
{y'_1  =  - y_1 w_3  + z_2 w_1,} & {y'_2  = y_2 w_3  - z_1 w_2 ,}
\\
{2z'_1  = x_1 w_2  - y_2 w_1,} & {2z'_2  =  - x_2 w_1 + y_1 w_2.}
\end{array}
\end{array}
\end{equation}
Here prime stands for $d/d(it)$.

Consider (\ref{eq3_7}) as the map $\mathbb{R}^9 \rightarrow
\mathbb{C}^9$ and denote its image by $V^9$. Eqs.\,(\ref {eq2_12})
of the phase space $P^6$ in $V^9$ take the form
\begin{eqnarray}
& z_1^2  + x_1 y_2  = r^2 ,\quad z_2^2  + x_2 y_1  = r^2 ,
\label{eq3_9}\\
& x_1 x_2  + y_1 y_2  + 2z_1 z_2  = 2p^2 .\label{eq3_10}
\end{eqnarray}
Here we introduce the positive constants
$$
p=\sqrt{\mathstrut a^2+b^2}, \quad r=\sqrt{\mathstrut a^2-b^2}.
$$

Using (\ref{eq3_9}) and (\ref{eq3_10}), express the first integrals
in new coordinates,
\begin{equation}\label{eq3_11}
\begin{array}{l}
\displaystyle{ H = w_1 w_2+ \frac{1}{2}w_3^2 - \frac{1}{2}(y_1 +
y_2)},  \\[2mm]
\displaystyle{
K=(w_1^2 + x_1 )(w_2^2  + x_2 )+2\lambda(w_1 w_2 w_3+z_2 w_1+z_1 w_2)-2\lambda^2 w_1 w_2},  \\[2mm]
\displaystyle{ G = \frac{1}{4}(p^2  - x_1 x_2 )w_3^2
+\frac{1}{2}(x_2 z_1 w_1 +x_1 z_2 w_2 )w_3  +}  \\[2mm]
\displaystyle{\phantom{ G =} + \frac{1}{4}(x_2 w_1  + y_1 w_2 )(y_2
w_1 + x_1 w_2 ) - \frac{1}{4}p^2 (y_1  + y_2 ) +}\\[2mm]
\displaystyle{\phantom{ G =}+\frac{1} {4}r^2 (x_1 + x_2
)+\frac{1}{2}\lambda(z_1 z_2 w_3+y_2 z_2 w_1 + y_1 z_1
w_2)+}\\[2mm]
\displaystyle{\phantom{ G =}+\frac{1}{4}\lambda^2(p^2-y_1 y_2).}
\end{array}
\end{equation}

Let $f$ be an arbitrary function on $V^9$. For brevity, the term
"critical point of $f$"\, will always mean a critical point of the
restriction of $f$ to $P^6$. Similarly, $df$ means the restriction
of the differential of $f$ to the set of vectors tangent to $P^6$.
While calculating critical points of various functions, it is
convenient to avoid introducing Lagrange's multipliers for the
restrictions (\ref{eq3_9}) and (\ref{eq3_10}).

\begin{lemma}\label{lem6}
Critical points of a function $f$ on $V^9$, in the above sense, are
defined by the system of equations
\begin{equation}\label{eq3_12}
X_i f=0 \quad (i=1,\dots 6),
\end{equation}
where
\begin{equation}\label{eq3_13}\notag
\begin{array}{l}
\displaystyle{X_1=\frac{\partial}{\partial w_1},
\;X_2=\frac{\partial}{\partial
w_2},\;X_3=\frac{\partial}{\partial w_3}, }\\[3mm]
\displaystyle{ X_4= z_2 \frac{\partial}{\partial x_2} + z_1
\frac{\partial}{\partial y_2} -\frac{1}{2}x_1 \frac{\partial}
{\partial z_1} - \frac{1}{2}y_1\frac{\partial}{\partial z_2}},\\[3mm]
\displaystyle{ X_5= z_1 \frac{\partial}{\partial x_1} + z_2
\frac{\partial}{\partial y_1}  - \frac{1}{2} y_2 \frac
{\partial }{ \partial z_1}-\frac{1}{2}x_2\frac{\partial}{\partial z_2}},\\[3mm]
\displaystyle{ X_6= x_1 \frac{\partial}{\partial x_1} - x_2
\frac{\partial}{\partial x_2} + y_1 \frac{\partial}{\partial y_1} -
y_2 \frac{\partial}{\partial y_2 }}.
\end{array}
\end{equation}
\end{lemma}
Indeed, six vector fields $X_i$ are tangent to $P^6$ and linearly
independent at any point of $P^6$.

The following two propositions define the strata $\mathfrak{C}_0$
and $\mathfrak{C}_1$ of the critical set.

\begin{proposition}\label{prop2}
The set $\mathfrak{C}_0$ consists exactly of the four equilibria
existing in this problem.
\end{proposition}
\begin{proof}
The condition of zero rank of the momentum map at a point $\zeta \in
P^6$ supposes, in particular, that $dH=0$. Then $\zeta$ is the point
of equilibrium and it follows from Lemma~\ref{lem3} that $\zeta$ is
one of the points (\ref{eq2_15}). Using the complex variables we
have
$$
\begin{array}{c}
w_1 = w_2 = w_3 =0, \quad z_1 = z_2 =0, \\
x_1 = x_2 = \varepsilon_1 a - \varepsilon_2 b, \quad y_1 = y_2 =
\varepsilon_1 a + \varepsilon_2 b \qquad (\varepsilon_1 = \pm 1,
\quad \varepsilon_2 = \pm 1).
\end{array}
$$
Use Eqs.\,(\ref{eq3_12}) with $f=K$ and $f=G$ to obtain that
$dK(\zeta)=0$ and $dG(\zeta)=0$. Therefore, $\rk J(\zeta)=0$.
\end{proof}

Note that in classical problems of the rigid body dynamics with an
axially symmetric force field, the rank of the momentum map is
everywhere not less than 1 due to the regularity of the cyclic
integral. In our case, all equilibria are {\it non-degenerate} (in
the Morse sense) critical points of the Hamilton function (see
\cite{KhZot}). Therefore, these points are critical for any first
integral of the system.

It is essential that in the sequel $\lambda \ne 0$.

\begin{proposition}\label{prop3}
The set $\mathfrak{C}_1$ is completely defined by the condition
$$
{\rk \{dK,dH\}=1}
$$
and consists of the points of the following
periodic trajectories:

1) pendulum motions $(\ref{eq2_18})$;

2) motions defined by the equations
\begin{equation}
\displaystyle{w_1 = q(w) \sqrt{w}, \quad w_2 =
\frac{\sqrt{w}}{q(w)},} \quad \displaystyle{w_3
=\frac{\lambda}{\sigma}w, } \label{eq3_14}
\end{equation}
\begin{equation}
\begin{array}{l} \displaystyle{ x_1 = \frac{1}{\sigma u}[\,r^2
\lambda^2 \sigma^2 -(\lambda^2+\sigma)\, u \, q^2(w) w\,], } \\[3mm]
\displaystyle{ x_2 = \frac{1}{\sigma u }[\,r^2 \lambda^2
\sigma^2 -(\lambda^2+\sigma)\, u \,\frac{w}{q^2(w)}\,], } \\[3mm]
\displaystyle{ y_1 = \sigma (1+\frac{\sigma}{\lambda^2}-\frac{r^4
\lambda^2 \sigma}{u^2})+\frac{r^2 \lambda^2}{u} q^2(w) w, } \\[3mm]
\displaystyle{ y_2 = \sigma (1+\frac{\sigma}{\lambda^2}-\frac{r^4
\lambda^2
\sigma}{u^2})+\frac{r^2 \lambda^2}{u} \frac{w}{q^2(w)}, } \\[3mm]
\displaystyle{ z_1 = -\frac{r^2 \lambda \, \sigma}{u}
\frac{\sqrt{w}}{q(w)}+\frac{\lambda^2+\sigma}{\lambda} q(w)
\sqrt{w}, } \\[3mm]
\displaystyle{ z_2 = -\frac{r^2 \lambda \, \sigma}{u}
q(w) \sqrt{w}+\frac{\lambda^2+\sigma}{\lambda}
\frac{\sqrt{w}}{q(w)}.}
\end{array}\label{eq3_15}
\end{equation}
Here $q(w)$ is the root of the equation $q^4-2Q(w)q^2+1=0$,
where
\begin{equation} \label{eq3_16} \displaystyle{Q(w) =
\frac{\sigma u^3+(\lambda^2+\sigma)[\lambda^2
w^2+\sigma^2(2w-\sigma)]u^2+r^4\lambda^4\sigma^4}{2r^2\lambda^2\sigma^2(\lambda^2+\sigma)u
w}};
\end{equation}
$\sigma, u$ are constants satisfying the equation
\begin{equation} \label{eq3_17}
\begin{array}{l}
\lambda^2(\lambda^2+\sigma)^2 u^5+(\lambda^2+\sigma)
[2p^2\lambda^4-(\lambda^2+\sigma)^3 \sigma]\sigma u^4+\\
\qquad +r^4\lambda^6\sigma^2 u^3+ 2 r^4 \lambda^4
\sigma^4(\lambda^2+\sigma)^2u^2-r^8\lambda^8 \sigma^6=0.
\end{array}
\end{equation}
The evolution $w(t)$ is defined by the equation
\begin{equation} \label{eq3_18}
\displaystyle{\big(\frac{dw}{dt}
\big)^2=-\frac{\lambda^2}{4\sigma^2} P_+ (w) P_- (w),}
\end{equation}
where
\begin{equation} \label{eq3_19}
\displaystyle{P_{\pm}(w) = w^2+2\sigma^2\frac{u \pm
r^2\lambda^2}{\lambda^2 u}w +
\frac{\sigma[u^3-(\lambda^2+\sigma)\sigma^2u^2+r^4\lambda^4
\sigma^3]}{(\lambda^2+\sigma)\lambda^2u^2}.}
\end{equation}
\end{proposition}
\begin{proof}
It follows from above that $dH \ne 0$ at the points of
$\mathfrak{C}_1$. Then to investigate the dependence of the
functions $K$ and $H$ it is sufficient to introduce the function
with one Lagrange's multiplier $\sigma$. Write Eqs.\,(\ref{eq3_12})
with $f=K- 2\sigma H$,
\begin{eqnarray}
& \begin{array}{l}
(w_1^2+x_1) w_2+ \lambda[z_1 + w_1(w_3-\lambda)]  -\sigma  w_1 =0,\\
(w_2^2+x_2) w_1+ \lambda[z_2 + w_2(w_3-\lambda)]  -\sigma w_2
=0,\label{eq3_20}
\end{array}\\
& \lambda w_1 w_2 - \sigma w_3=0,\label{eq3_21}\\
& \begin{array}{l}
(w_1^2+x_1) z_2 - \lambda (w_2 x_1+w_1 y_1) +\sigma z_1=0,\\
(w_2^2+x_2) z_1 - \lambda (w_1 x_2+w_2 y_2) +\sigma z_2=0,
\end{array} \label{eq3_22}\\
& x_1 w_2^2 - x_2 w_1^2+\sigma (y_1-y_2) =0.\label{eq3_23}
\end{eqnarray}

First consider the critical points of the function $K$. For this
purpose we must put $\sigma=0$. Eq.\,(\ref{eq3_21}) gives $w_1 = w_2
=0$. Then Eqs.\,(\ref{eq3_20}) imply $z_1 = z_2 =0$.
Eqs.\,(\ref{eq3_22}) and (\ref{eq3_23}) become identities. The same
values satisfy Eqs.\,(\ref{eq3_12}) if we take $f=4G+(x_1 x_2-y_1
y_2)H$. Therefore, $dK=0$ and $4dG+(x_1 x_2-y_1 y_2)dH=0$. Since $dH
\ne 0$, it means that $\rk J=1$. The initial variables on the
corresponding trajectories are $\omega_1 = \omega_2 \equiv 0$,
$\alpha_3 = \beta_3 \equiv 0$. Substitute these values to
Eqs.\,(\ref{eq3_1}) to obtain the solutions~(\ref{eq2_18}).

Let ${\sigma \ne 0}$. The equilibria of the system are already
excluded. Then it follows from~(\ref{eq3_21}) that $ w_1 w_2 \ne 0$.
Satisfying (\ref{eq3_21}), introduce new variables $w,q$ as shown
in~(\ref{eq3_14}). Four equations~(\ref{eq3_20}), (\ref{eq3_22})
form the linear system in $y_1, y_2, z_1, z_2$, from which we obtain
these variables as the functions of $x_1, x_2, w, q$ identically
satisfying~(\ref{eq3_23}). Denote
\begin{equation} \label{eq3_24}
u={(w-\sigma)^2(\lambda^2+\sigma)-\sigma x_1 x_2}.
\end{equation}
Then Eqs.\,(\ref{eq3_9}) are easily solved for $x_1,x_2$ as the
functions of $w,q,u$. As a result we obtain the expressions
(\ref{eq3_15}). Let
$$
\displaystyle{Q=\frac{1}{2}(q^2+\frac{1}{q^2})}.
$$
Then the substitution of $x_1,x_2$ from (\ref{eq3_15}) back to
(\ref{eq3_24}) gives (\ref{eq3_16}). The last unused equation
(\ref{eq3_10}) provides the relation (\ref{eq3_17}) between $u$ and
the constants $\lambda, \sigma$. It shows that $u$ defined
as~(\ref{eq3_24}) appears to be a constant.

Thus, all phase variables are expressed via one variable $w$, for
which from (\ref{eq3_8}) we find the differential
equation~(\ref{eq3_18}). Note that due to (\ref{eq3_19}) the
solutions are elliptic functions of time.

To finish the proof, we need to show that at the points of the
trajectories found we really have $\rk J=1$, i.e., the linear
dependence of $dK$ and $dH$ implies the linear dependence of $dG$
and $dH$. Indeed, Eqs.\,(\ref{eq3_12}) with
$$
f=2 G - (p^2+\frac{\lambda^2+\sigma}{\lambda^2 \sigma}u) H
$$
are satisfied both by (\ref{eq2_18}) and by (\ref{eq3_14}),
(\ref{eq3_15}). Therefore, $\rk \{dG, dH \}=1$ and, consequently,
$\rk \{dK, dG, dH \}=1$.
\end{proof}

The following statement describes one of the critical subsystems in
$\mathfrak{C}_2$.
\begin{proposition}\label{prop4}
The system $(\ref{eq3_8})$ has the four-dimensional invariant
submanifold $\mathfrak{O}_*$ defined by the equations
\begin{equation}\label{eq3_25}
U_1=0, \quad U_2=0,
\end{equation}
where
\begin{equation}\label{eq3_26}
\begin{array}{l}
\displaystyle{U_1=\frac{y_2 w_1+x_1
w_2+z_1(w_3+\lambda)}{w_1}-\frac{x_2 w_1+y_1
w_2+z_2(w_3+\lambda)}{w_2},}\\[2mm]
\displaystyle{U_2=w_1 w_2 U'_1.}\\
\end{array}
\end{equation}
The Poisson bracket $\{U_1,U_2\}$ is non-zero almost everywhere on
this submanifold.
\end{proposition}
\begin{proof} The derivative $U'_2$ in virtue of
(\ref{eq3_8}) is proportional to $U_1$, i.e., (\ref{eq3_25}) implies
$U'_2=0$. Therefore the set (\ref{eq3_25}) is invariant.

Consider the function
$$
\displaystyle{S=-\frac{1}{4}\left[\frac{y_2 w_1+x_1
w_2+z_1(w_3+\lambda)}{w_1}+\frac{x_2 w_1+y_1
w_2+z_2(w_3+\lambda)}{w_2}\right].}
$$
On $\mathfrak{O}_*$ we obtain
$$
\displaystyle{S'=-\frac{w_1 z_2 + w_2 z_1 + w_1 w_2 (w_3 -
\lambda)}{8 w_1 w_2} U_1 \equiv 0.}
$$
Therefore, $S$ is a partial integral of the induced system.
Eliminate $y_1,y_2$ with the help of Eqs.\,(\ref{eq3_25}) and
present $S$ in a more simple form
\begin{equation}\label{eq3_27}
\displaystyle{ S = \frac{x_2 z_1 w_1+x_1 z_2 w_2+z_1 z_2(w_3+\lambda
)}{2 w_1 w_2 (w_3-\lambda)}.}
\end{equation}
Now the Poisson bracket of $U_1$ and $U_2$ calculated under the
rules defined by (\ref{eq2_3}) is expressed in terms of the energy
constant $h$ and the constant $s$ of the integral (\ref{eq3_27}) in
the following way
\begin{equation}\label{eq3_28}\notag
\displaystyle{ \{U_1,U_2\} = -\frac{4}{s}\left[3 s^4 -2
s^3(h-\frac{\lambda^2}{2})+\frac{p^4-r^4}{4}\right].}
\end{equation}
Obviously, the right part of it is a ratio of polynomials not
identically zero on~$\mathfrak{O}_*$. Therefore the set
$\{U_1,U_2\}=0$ has codimension 1 in $\mathfrak{O}_*$. In particular
this set is of zero measure in $\mathfrak{O}_*$.
\end{proof}
\begin{remark}\label{rem3}
If $\lambda=0$, then the manifold $\mathfrak{O}_*$ turns into the
phase space of the Hamiltonian system with two degrees of freedom
studied in {\rm \cite{KhRCD2}}. The geometrical characteristic of
the motions in this system is the condition
$$
\displaystyle{\frac{{\bf M} {\bs \cdot} {\bs \alpha}}{{\bf M} {\bs
\cdot} {\bf e}_1}=\frac{{\bf M} {\bs \cdot} {\bs \beta}}{{\bf M}
{\bs \cdot} {\bf e}_2}={\rm const},}
$$
where ${\bf M}={\bf I}{\bs \omega}$ is the angular momentum vector.
The system $(\ref{eq3_25})$, $(\ref{eq3_26})$ is found from the same
condition given that here ${\bf M}={\bf I}{\bs \omega}+\bl$.
\end{remark}

The following theorem completes the description of the critical set
of the momentum map for the gyrostat.

\begin{theorem} \label{th2}
The set of critical points of the momentum map~$(\ref{eq3_4})$
consists of the following subsets in~$P^6$:

1) the set $\mathfrak{L}$ defined by the system
\begin{equation}\label{eq3_29}
w_1  = 0, \quad w_2  = 0, \quad z_1  = 0, \quad z_2  = 0;
\end{equation}

2) the set $\mathfrak{N}$ defined by the system
\begin{equation}\label{eq3_30}
F_1  = 0,\quad F_2  = 0,
\end{equation}
where
$$
\begin{array}{l}
F_1 = (w_1 w_2+\lambda w_3)(w_2 x_1+\lambda
z_1)\lambda y_1 -\\[1.5mm]
\qquad -w_2(w_1^2+x_1)(x_2 z_1 w_1+x_1 z_2 w_2-x_1 x_2 w_3+2 z_1 z_2 \lambda)-\\[1.5mm]
\qquad -x_2(w_1 w_3+z_1)(w_1 z_1-x_1 w_3)\lambda+(x_1 w_3^2-2 z_1 w_1 w_3-z_1^2)z_2 \lambda^2,\\[1.5mm]
F_2 = (w_1 w_2+\lambda w_3)(w_1 x_2+\lambda
z_2)\lambda y_2 -\\[1.5mm]
\qquad -w_1(w_2^2+x_2)(x_2 z_1 w_1+x_1 z_2 w_2-x_1 x_2 w_3+2 z_1 z_2 \lambda)-\\[1.5mm]
\qquad -x_1(w_2 w_3+z_2)(w_2 z_2-x_2 w_3)\lambda+(x_2 w_3^2-2 z_2
w_2 w_3-z_2^2)z_1 \lambda^2;
\end{array}
$$

3) the set $\mathfrak{O}$ defined by the system
\begin{equation}\label{eq3_31}
R_1  = 0,\quad R_2  = 0,
\end{equation}
where
$$
\begin{array}{l}
\displaystyle{R_1= [y_1 w_2+ x_2 w_1 +z_2 (w_3+\lambda)]w_1
 (w_3-\lambda ) + }\\[1.5mm]
\displaystyle{\phantom{R_1=} +x_2 z_1 w_1+x_1 z_2 w_2+z_1
z_2(w_3+\lambda ),}
 \\[1.5mm]
\displaystyle{R_2= [y_2 w_1+ x_1 w_2 +z_1 (w_3+\lambda)]w_2
 (w_3-\lambda ) +}\\[1.5mm]
\displaystyle{\phantom{R_1=} + x_2 z_1 w_1+x_1 z_2 w_2+z_1
z_2(w_3+\lambda ).}
 \end{array}
$$
\end{theorem}

\begin{proof}
We need to prove that
\begin{equation}\label{eq3_32}
\mathfrak{L}\cup\mathfrak{N}\cup\mathfrak{O}=\mathfrak{C}.
\end{equation}
It follows from Propositions~\ref{prop2} and \ref{prop3} that
$\mathfrak{L} \subset \mathfrak{C}_0 \cup \mathfrak{C}_1$. Indeed on
$\mathfrak{L}$ we have $ dK \equiv 0$, ${dG \equiv \pm a b \, dH}$.
Note also that the system of relations (\ref{eq3_29}) satisfies the
conditions of Lemma~\ref{lem5} with $n=3,k=2$. Therefore,
$\mathfrak{L}$ is a smooth two-dimensional manifold with the induced
Hamiltonian system with one degree of freedom.

According to Proposition~\ref{prop4} we have $\mathfrak{O}_* \subset
\mathfrak{C}_2$. Hence, $\mathfrak{O}=\mathfrak{O}_* \cup
\mathfrak{L} \subset \mathfrak{C}$.

Consider the set $\mathfrak{N}$ defined by Eqs.\,(\ref{eq3_30}).
First, investigate the cases when these equations cannot be solved
with respect to $y_1, y_2$. Suppose that
\begin{equation}\label{eq3_33}
w_1 w_2+\lambda w_3 \equiv 0.
\end{equation}
Then, after several differentiations in virtue of (\ref{eq3_8}), we
come to Eqs.\,(\ref{eq3_20}) -- (\ref{eq3_23}) with
${\sigma=-\lambda^2}$. The corresponding points belong to
$\mathfrak{C}_1$. Let
\begin{equation}\label{eq3_34}
(w_2 x_1+\lambda z_1)(w_1 x_2+\lambda z_2) \equiv 0.
\end{equation}
Then the same procedure leads to the system of equations having the
only solutions of the form (\ref{eq3_29}), i.e., to the set
$\mathfrak{L}$. Denote $\mathfrak{N}_* = \mathfrak{N} \backslash
(\mathfrak{C}_0 \cup \mathfrak{C}_1)$. On this set from
(\ref{eq3_30}) we obtain
\begin{equation}\label{eq3_35}
\begin{array}{l}
\displaystyle{y_1 = \frac{1}{(w_1 w_2+\lambda w_3)(w_2 x_1+\lambda z_1)\lambda}
[w_2(w_1^2+x_1)(x_2 z_1 w_1+x_1 z_2 w_2-}\\[3mm]
\qquad -x_1 x_2 w_3+2 z_1 z_2 \lambda)+x_2(w_1 w_3+z_1)(w_1 z_1-x_1
w_3)\lambda - \\[1.5mm]
\qquad -(x_1 w_3^2-2 z_1 w_1 w_3-z_1^2)z_2 \lambda^2],\\[1.5mm]
\displaystyle{y_2 = \frac{1}{(w_1 w_2+\lambda w_3)(w_1 x_2+\lambda
z_2)\lambda} [w_1(w_2^2+x_2)(x_2 z_1 w_1+x_1 z_2 w_2-}\\[3mm]
\qquad -x_1 x_2 w_3+ 2 z_1 z_2 \lambda)+x_1(w_2 w_3+z_2)(w_2 z_2-x_2
w_3)\lambda -\\[1.5mm]
\qquad -(x_2 w_3^2-2 z_2 w_2 w_3-z_2^2)z_1 \lambda^2].
\end{array}
\end{equation}
The derivatives of $F_1$ and $F_2$ in virtue of (\ref{eq3_8}) vanish
identically after the substitution of the expressions (\ref{eq3_35}).
This fact proves that $\mathfrak{N}_*$ is an invariant set. The
Poisson bracket $\{F_1, F_2\}$ with~(\ref{eq3_35}) takes the form
\begin{equation}\label{eq3_36}
\displaystyle{\{F_1, F_2\}=\sqrt{2}\,\lambda(w_1 w_2+\lambda
w_3)^{3/2} \sqrt{(w_2 x_1+\lambda z_1)(w_1 x_2+\lambda z_2)} }\, C.
\end{equation}
Here
\begin{equation}\label{eq3_37}
C=\displaystyle{\frac{1}{s}(8s^3\lambda^2-r^4)\sqrt{2s^2-(2h+\lambda^2)s+p^2}}
\end{equation}
depends on the energy constant $h$ and the constant $s$ of the
partial integral
\begin{equation}\label{eq3_38}
\displaystyle{ S= \frac{x_1 x_2 w_3-x_2 z_1 w_1-x_1 z_2 w_2-\lambda
z_1 z_2}{2\lambda(w_1 w_2+\lambda w_3)}.}
\end{equation}
This integral is similar to (\ref{eq3_27}) and independent of $H$
almost everywhere on $\mathfrak{N}_*$.

The cases when the multipliers (\ref{eq3_33}) or (\ref{eq3_34}) in
(\ref{eq3_36}) turn to zero are already studied above. Zeros of the
function (\ref{eq3_37}) have codimension 1. Hence, (\ref{eq3_36}) is
non-zero almost everywhere on~$\mathfrak{N}_*$. It follows from
Lemma~\ref{lem5} that $\mathfrak{N}_* \subset \mathfrak{C}_2$. Thus,
$\mathfrak{L}\cup\mathfrak{N}\cup\mathfrak{O}\subset \mathfrak{C}$.
To prove the equality (\ref{eq3_32}), we must show that
$\mathfrak{C}\subset\mathfrak{L}\cup\mathfrak{N}\cup\mathfrak{O}$.

The points of the set $\mathfrak{C}_0$ described in
Proposition~\ref{prop2} satisfy (\ref{eq3_29}). According to
Proposition~\ref{prop3} the set $\mathfrak{C}_1$ can be represented
as $\mathfrak{C}_{11} \cup \mathfrak{C}_{12}$, where
$\mathfrak{C}_{11}$ consists of the trajectories (\ref{eq2_18}) and
$\mathfrak{C}_{12}$ is defined by the system
(\ref{eq3_14})--(\ref{eq3_17}). On the trajectories (\ref{eq2_18})
we have (\ref{eq3_29}). It is easily checked that the points given by
Eqs.\,(\ref{eq3_14})--(\ref{eq3_16}) under the condition
(\ref{eq3_17}) satisfy both systems (\ref{eq3_30}) and (\ref{eq3_31}).
Therefore, $\mathfrak{C}_0 \cup \mathfrak{C}_{11}
\subset\mathfrak{L}$ and $\mathfrak{C}_{12}
\subset\mathfrak{N}\cap\mathfrak{O}$.

Consider now the set $\mathfrak{C}_2$. To investigate the dependence
of $G,H,K$, introduce the function with Lagrange's multipliers. It
follows from Propositions~\ref{prop2} and~\ref{prop3} that on
$\mathfrak{C}_2$ the differentials $dK$ and $dH$ are linearly
independent. Then the multiplier at the function $G$ is always
non-zero and can be chosen equal to any non-zero constant. It is
convenient to take the function $2 G + S K + (T - p^2)H$, where $S$
and $T$ are Lagrange's undefined multipliers. The condition
\begin{equation}\label{eq3_39}
2 dG + S \, dK + (T - p^2)\,dH=0
\end{equation}
is preserved by the phase flow. Applying the Lie derivative we
obtain
$$
\dot{S} \, dK + \dot{T} \,dH = 0.
$$
Since $\rk \{dG,dK,dH\}=2$, this linear combination of the
differentials is proportional to the left part of (\ref{eq3_39}). It
means that, at the points of~$\mathfrak{C}_2$,
$$
\dot{S} \equiv 0,\quad \dot{T} \equiv 0.
$$
Thus, the functions $S$ and $T$ are the partial integrals on the
submanifold $\mathfrak{C}_2$. According to Lemma~\ref{lem6} rewrite
Eq.\,(\ref{eq3_39}) as the system
\begin{equation}\label{eq3_40}
\begin{array}{l}
x_2(y_2+2S)w_1+2 S(w_1 w_2+\lambda w_3)w_2+\\
\qquad+(T-z_1 z_2-2S \lambda^2)w_2+x_2 z_1 w_3+(y_2+2S)z_2
\lambda=0,
\\
x_1(y_1+2S)w_2+2 S(w_1 w_2+\lambda w_3)w_1+\\
\qquad +(T-z_1 z_2-2S \lambda^2)w_1+x_1 z_2 w_3+(y_1+2S)z_1
\lambda=0,
\end{array}
\end{equation}
\begin{equation}\label{eq3_41}
\begin{array}{l} (T- x_1 x_2)w_3+x_2 z_1 w_1+x_1 z_2 w_2+(2S w_1
w_2+z_1 z_2)\lambda=0,\\[1mm]
T z_1+ x_1 z_2 w_3^2 +[(x_1 x_2 - 2 z_1 z_2)w_1 + (y_1 z_1 + x_1
z_2)\lambda + x_1 y_1 w_2]w_3-\\
\qquad - (y_1 z_1 + x_1 z_2)w_1 w_2 + x_1 (y_1 + 2S)w_2 \lambda -
\\
\qquad -[x_2 z_1 + (y_2 + 2S)z_2]w_1^2 + [(y_2 + 2S)y_1 - 2z_1
z_2]w_1 \lambda + y_1 z_1 \lambda^2 - \\
\qquad -[(y_2 + 2S)x_1 + z_1^2]z_2=0,\\[1mm]
T z_2+x_2 z_1 w_3^2 +[(x_1 x_2 - 2 z_1 z_2)w_2 + (y_2 z_2 + x_2
z_1)\lambda + x_2 y_2 w_1]w_3-\\
\qquad - (y_2 z_2 + x_2 z_1)w_1 w_2 + x_2 (y_2 +
2S)w_1 \lambda - \\
\qquad -[x_1 z_2 + (y_1 + 2S)z_1]w_2^2 + \\
\qquad+ [(y_1 + 2S)y_2 - 2z_1 z_2]w_2 \lambda + y_2 z_2 \lambda^2
-\\
\qquad -[(y_1 + 2S)x_2 + z_2^2]z_1 =0,\\[1mm]
(T - x_1 x_2) ( y_1 - y_2) + 2 ( y_2+S) x_2 w_1^2  -2(y_1+S)x_1 w_2^2 + \\
\qquad+2  (x_2 z_1  w_1 - x_1 z_2 w_2 ) w_3 + x_2 z_1^2 - x_1 z_2^2
+ \\
\qquad +2(y_2 z_2 w_1 -y_1 z_1 w_2  )\lambda=0.
\end{array}
\end{equation}

It follows from Proposition~\ref{prop4} that this system is valid at
the points of the set $\mathfrak{O}_*$. To find all other cases
suppose
\begin{equation}\label{eq3_42}
U_1 U_2 \neq 0
\end{equation}
and express $y_1,y_2$ from (\ref{eq3_26}):
\begin{equation}\label{eq3_43}
\begin{array}{l}
\displaystyle{y_1= \frac{1}{2w_1 w_2 (w_3-\lambda)}\{2 U_2 - [w_1
w_2(w_3-\lambda)+w_2 z_1-w_1 z_2] U_1-}\\[3mm]
\phantom{y_1=} -2[w_1 z_2(w_3-\lambda)^2+ (x_2 w_1^2+z_1 z_2 +2
\lambda w_1
z_2)(w_3-\lambda)+\\
\phantom{y_1=} + x_2 z_1 w_1+x_1 z_2
w_2+2\lambda z_1 z_2]\},\\[2mm]
\displaystyle{y_2= \frac{1}{2w_1 w_2 (w_3-\lambda)}\{2 U_2 + [w_1
w_2(w_3-\lambda)+w_1 z_2-w_2 z_1] U_1-}\\[3mm]
\phantom{y_2=}-2[w_2 z_1(w_3-\lambda)^2+ (x_1 w_2^2+z_1 z_2 +2
\lambda w_2
z_1)(w_3-\lambda)+\\
\phantom{y_2=} +x_2 z_1 w_1+x_1 z_2 w_2+2\lambda z_1 z_2]\}.
\end{array}
\end{equation}

The determinant of the system (\ref{eq3_40}) with respect to $T,2S$
is equal to
$$
\Delta = x_1 w_2^2-x_2 w_1^2-(z_2 w_1- z_1 w_2) \lambda.
$$
If we suppose that $\Delta \equiv 0$ on some time interval, then the
sequence of the derivatives of this identity in virtue of
(\ref{eq3_8}) leads to (\ref{eq3_29}), i.e., to the points of
$\mathfrak{C}_0\cup\mathfrak{C}_1$. Consider then
\begin{equation}\label{eq3_44}
\Delta \neq 0
\end{equation}
and find from Eqs.\,(\ref{eq3_40})
\begin{equation}\label{eq3_45}
\begin{array}{l}
\displaystyle{S=\frac{1}{\Delta}[x_2 y_2 w_1^2  - x_1 y_1 w_2^2 + (
x_2 z_1 w_1 - x_1 z_2 w_2 )w_3 +}\\[2mm]
\qquad +(y_2 z_2 w_1 -y_1 z_1 w_2 )\lambda],\\[2mm]
\displaystyle{T=\frac{1}{\Delta}[A_1 B_1 - A_2 B_2].}
\end{array}
\end{equation}
Here
\begin{equation}\label{eq3_46}\notag
\begin{array}{l}
A_1 =(x_1 w_2+\lambda z_1)y_1+(x_1 w_3-z_1 w_1)z_2, \\
B_1=(w_2^2+x_2)w_1+\lambda w_2(w_3-\lambda)+\lambda z_2,\\
A_2=(x_2 w_1+\lambda z_2)y_2+(x_2 w_3-z_2 w_2)z_1,\\
B_2=(w_1^2+x_1)w_2+\lambda w_1(w_3-\lambda)+\lambda z_1.
\end{array}
\end{equation}

Substitute (\ref{eq3_43}) and (\ref{eq3_45}) into (\ref{eq3_41}) to
obtain the system of four equations of the type $E_i =0$
($i=1,...,4$), where $E_i = a_{i2} U_1^2+a_{i1} U_1+a_{i0} U_2$ with
some polynomials $a_{ij}$. For the proof of the theorem, there is no
need to use all equations of this system.  It is enough to consider,
for example, the zero points of the resultant of two simplest
functions $E_1, E_4$ with respect to $U_2$. We obtain ${4 w_1 w_2
U_1 R \Delta =0}$, where
$$
\begin{array}{l}
R = \lambda w_1 w_2 ( x_1 w_2+\lambda z_1) ( x_2 w_1+\lambda z_2)
U_1-\big\{ w_1 w_2[x_2 z_1 w_1+ x_1 z_2 w_2- \\
\phantom{Q = } -x_1 x_2 (w_3-\lambda)+2 \lambda z_1 z_2]+(z_1 z_2
w_3+ x_2 z_1 w_1+x_1 z_2 w_2)\lambda^2+z_1 z_2\lambda^3
\big\}\Delta.
\end{array}
$$
Due to (\ref{eq3_42}), $U_1 \neq 0$. At the points of $\mathfrak{C}
\backslash \mathfrak{L}$ the product $w_1 w_2$ is not identically
zero. The set $\mathfrak{C} \backslash
(\mathfrak{L}\cup\mathfrak{O})$ is, obviously, preserved by the
phase flow. Therefore, (\ref{eq3_44}) and (\ref{eq3_8}) imply
$R=0,\,R'=0$. These equations are linear in $y_1,y_2$ and yield the
expressions (\ref{eq3_35}) satisfying Eqs.\,(\ref{eq3_30}). Thus,
$\mathfrak{C} \backslash (\mathfrak{L}\cup\mathfrak{O}) \subset
\mathfrak{N}$.
\end{proof}

\begin{remark}\label{rem4}
The system~$(\ref{eq3_31})$ follows from Eqs.\,$(\ref{eq3_25})$ and
$(\ref{eq3_26})$. In particular,
$\mathfrak{O}_*=\mathfrak{O}\backslash \mathfrak{L}$. Then from
Proposition~\ref{prop4} and Lemma~\ref{lem5} we obtain that this
manifold lies completely in~$\mathfrak{C}_2$.
\end{remark}

\begin{remark}\label{rem5}
The integral $S$ given by $(\ref{eq3_45})$, when restricted to the
set $\mathfrak{N}$ has the form $(\ref{eq3_38})$. Indeed, it is
enough to substitute $(\ref{eq3_35})$ into the first formula
$(\ref{eq3_45})$ to obtain $(\ref{eq3_38})$. On the set
$\mathfrak{O}$ the same function $S$ in the substitution of
$(\ref{eq3_25})$ and $(\ref{eq3_43})$ takes the form
$(\ref{eq3_27})$. Therefore, the use of the same notation in
$(\ref{eq3_27})$ and $(\ref{eq3_38})$ is correct. The expressions
for $T$ can also be simplified. On the set $\mathfrak{N}$ we have
$T=2\lambda^2 S$, i.e., this function does not give rise to a new
partial integral independent of $S$. On the contrary, at the points
of the set $\mathfrak{O}$ we have
\begin{equation}\label{eq3_47}\notag
T = x_1 x_2 + z_1 z_2 -2 w_1 w_2 S.
\end{equation}
In the case $\lambda=0$ the same expression with the corresponding
function $S$ is the partial integral independent of $S$. The
equations of the integral manifold defined by the pair $S,T$ lead to
the separation of variables on $\mathfrak{O}$ {\rm \cite{KhRCD2}}.
\end{remark}

\section{The bifurcation diagram}\label{sec4}
The Lax representation for the considered problem found
in~\cite{ReySem} can be written in the form
\begin{equation} \label{eq4_1}
\displaystyle{L' = L M - M L},
\end{equation}
where
\begin{equation} \label{eq4_2}\notag
\begin{array}{l}
\displaystyle{L=\begin{Vmatrix} \displaystyle{2 \lambda} &
\displaystyle{\frac{x_2}{\varkappa}} & \displaystyle{-2 w_1} &
\displaystyle{\frac{z_2}{\varkappa}}
\\[2mm]
\displaystyle{-\frac{x_1}{\varkappa}} & \displaystyle{-2 \lambda} &
\displaystyle{-\frac{z_1}{\varkappa}} &
\displaystyle{2 w_1 } \\[2mm]
\displaystyle{-2 w_1} & \displaystyle{\frac{z_2}{\varkappa}} &
\displaystyle{-2 w_3} &
\displaystyle{-\frac{y_1}{\varkappa} - 4 \varkappa} \\[2mm]
\displaystyle{-\frac{z_1}{\varkappa}} & \displaystyle{2 w_2} &
\displaystyle{\frac{y_2}{\varkappa}+ 4\varkappa} & \displaystyle{2
w_3}
\end{Vmatrix}},\;
\displaystyle{M=\begin{Vmatrix} \displaystyle{-\frac{w_3}{2}} &
\displaystyle{0} & \displaystyle{\frac{w_2}{2}} &
\displaystyle{0}\\[2mm]
\displaystyle{0} & \displaystyle{\frac{w_3}{2}} & \displaystyle{0} &
\displaystyle{-\frac{w_1}{2}}\\[2mm]
\displaystyle{\frac{w_1}{2}} & \displaystyle{0} &
\displaystyle{\frac{w_2}{2}} &
\displaystyle{\varkappa}\\[2mm]
\displaystyle{0} & \displaystyle{-\frac{w_2}{2}} &
\displaystyle{-\varkappa} & \displaystyle{-\frac{w_3}{2}}
\end{Vmatrix}.}
\end{array}
\end{equation}
Here $\varkappa$ stands for the spectral parameter, the derivative
in~(\ref{eq4_1}) is calculated in virtue of the
system~(\ref{eq3_8}). The equation for the eigenvalues $\mu$ of the
matrix $L$ defines the algebraic curve associated with this
representation~\cite{BobReySem}. Let $s=2\varkappa^2$ and let
$h,k,g$ be the arbitrary constants of the integrals~(\ref{eq3_11}).
The equation of the algebraic curve takes the form
\begin{equation} \label{eq4_3}
\begin{array}{l}
\displaystyle{ \mu ^4 -  4   \mu ^2   [ \frac{p ^2}{s}  - (2  h +
\lambda ^2) + 2  s ] + 4   [\frac{r ^4}{s^2} + \frac{2}{s}(4  g -
2p^2 h-p^2\lambda ^2) + }\\[2mm]
\qquad +4(k + 2\lambda ^2 h)
 - 8\lambda ^2 s ] = 0.
\end{array}
\end{equation}

It is natural to suppose that the bifurcation diagram of the
momentum map~(\ref{eq3_4}) is included in the set of
values~$(g,k,h)$ such that the curve~(\ref{eq4_3}) either have
singular points or is reducible, i.e., the left part of
Eq.\,(\ref{eq4_3}) splits into the product of some rational
non-trivial expressions. In this way we can guess the result of the
following statement. Nevertheless, to obtain the complete proof of
it, we must fulfil the calculations on the above found critical
manifolds.

\begin{theorem}\label{th3}
The bifurcation diagram of the momentum map $G\times K \times H$ is
included in the union of the following (intersecting) subsets of
${\mathbb{R}}^3(g,k,h)$:

1) the pair of straight lines
\begin{equation} \label{eq4_4}
\Gamma_+: \left\{ \begin{array}{l} k = (a + b)^2, \\
\displaystyle{g= - \,a b (h-\frac{\lambda^2}{2});}
\end{array} \right. \qquad \Gamma_-: \left\{ \begin{array}{l} k = (a - b)^2, \\
\displaystyle{g=  \,a b (h-\frac{\lambda^2}{2});}
\end{array} \right.
\end{equation}

2) the surface
\begin{equation} \label{eq4_5}
\Gamma_1: \left\{ \begin{array}{l}
\displaystyle{k = 4 \lambda^2 s - 2\lambda^2 h + \frac{r^4}{4s^2},} \\[3mm]
\displaystyle{g=-\lambda^2 s^2 +
\frac{1}{2}p^2(h+\frac{\lambda^2}{2})-\frac{r^4}{4s}, \qquad s \in
{\mathbb{R}}\backslash\{0\};}
\end{array} \right.
\end{equation}

3) the surface
\begin{equation} \label{eq4_6}
\Gamma_2: \left\{ \begin{array}{l} \displaystyle{k = 3s^2
-4(h-\frac{\lambda^2}{2})s+ p^2 + (h-\frac{\lambda^2}{2})^2-
\frac{p^4-r^4}{4s^2},} \\[3mm]
\displaystyle{g=-s^3+(h-\frac{\lambda^2}{2})s^2+\frac{p^4-r^4}{4s},
\qquad s \in {\mathbb{R}}\backslash\{0\}.}
\end{array} \right.
\end{equation}
\end{theorem}
\begin{proof}
Let $\zeta \in \mathfrak{L}$. Substitution of the values $z_1=z_2=0$
into (\ref{eq3_9}) and (\ref{eq3_10}) yields $x_1 x_2 =(a \pm
b)^2$, $y_1 y_2 =(a \mp b)^2$. Then from (\ref{eq3_11}),
(\ref{eq3_29}) we obtain the equations defining the lines
(\ref{eq4_4}).

Let $\zeta \in \mathfrak{N}\backslash \mathfrak{L}$. Take the
constant of the partial integral~(\ref{eq3_38}) for the parameter $s$
in (\ref{eq4_5}), substitute the expressions (\ref{eq3_11}) for the
corresponding constants, and fulfil the change (\ref{eq3_35}). Then
both Eqs.\,(\ref{eq4_5}) become the identities. Therefore,
$J(\mathfrak{N}\backslash \mathfrak{L}) \subset \Gamma_1$.

The inclusion $J(\mathfrak{O}\backslash \mathfrak{L}) \subset
\Gamma_2$ is proved in a similar way. We take the constant of the
partial integral (\ref{eq3_27}) for the parameter $s$ in
(\ref{eq4_6}) and fulfil the substitution (\ref{eq3_43}) with
$U_1=U_2=0$.
\end{proof}
\begin{remark}\label{rem6}
Note that the shift of the energy level $\tilde h = h-\lambda^2/2$
makes the equations of the lines $\Gamma_{\pm}$ and the surface
$\Gamma_2$ independent of $\lambda$. Thereby obtained equations are
identical with the corresponding equations of the case $\lambda=0$
{\rm \cite{KhRCD1}}. The surface $\Gamma_1$ is obtained as a
perturbation (with respect to $\lambda$) of two tangent to each
other sheets of the bifurcation diagram of the case $\lambda=0$,
i.e., the plane $k=0$ and the slanted parabolic cylinder $(p^2
h-2g)^2-r^4 k =0$. Thus, it is easy to view the evolution of the
Appelrot classes {\rm \cite{Appel}} of the S.\,Kowalevski case in
the process of two-way generalizations---adding the second force
field and, afterwards, the non-zero gyrostatic momentum.
\end{remark}

The equations given in Theorem~\ref{th3} are in the following sense
convenient. Let us fix the energy constant $h$. Then we obtain the
parametric equations of a one-dimensional set in the plane $(g,k)$
(with the finite number of singular points). This set is the
bifurcation diagram $\Sigma_h$ of the restriction of the pair of
integrals $G,K$ onto the iso-energetic surface $\{H=h\}\subset P^6$,
which is always compact. In particular, all diagrams $\Sigma_h$ lie
in the restricted area of the $(g,k)$-plane and are easily drawn
numerically. The analytical investigation of the types of the
diagrams $\Sigma_h$ with respect to the essential parameters
$(b/a,\lambda/\sqrt{a}, h/a)$ is a necessary but technically
complicated problem. Nevertheless, it must be solvable. Indeed, the
set of double points and cusps of the curves $\Gamma_{1,2}$ in the
$(g,k)$-plane is easily defined and investigated analytically.
Moreover, the values of the first integrals on the periodic motions
(\ref{eq3_14})--(\ref{eq3_17}) define the points of transversal
intersections $\Gamma_1 \cap \Gamma_2$. This fact, at least,
guarantees that the numerical algorithm can be built for effective
calculation of knots of one-dimensional cell complex $\Sigma_h$ for
any $h$. In turn, it should be possible to find all cases of
bifurcations of the set of these knots with respect to the
parameters defining the above set of periodic motions.


\begin{thebibliography}{99}

\bibitem{Kowa} S.~Kowalevski, Sur le probleme de la rotation d'un
corps solide autour d'un point fixe, {\it Acta Math.}, {\bf 12}
(1889), 177-232.

\bibitem{ReySem} A.G.~Reyman, M.A.~Semenov-Tian-Shansky, Lax
representation with a spectral parameter for the Kowalewski top and
its generalizations, {\it Lett. Math. Phys.}, {\bf 14}, 1 (1987),
55-61.

\bibitem{Bogo} O.I.~Bogoyavlensky, Euler equations on
finite-dimension Lie algebras arising in physical problems, {\it
Commun. Math. Phys.}, {\bf 95} (1984), 307-315.

\bibitem{Yeh} H.~Yehia, New integrable cases in the dynamics of
rigid bodies, {\it Mech. Res. Commun.}, {\bf 13}, 3 (1986), 169-172.

\bibitem{Komar} I.V.~Komarov, A generalization of the Kovalevskaya
top, {\it Phys. Letters}, {\bf 123}, 1 (1987), 14-15.

\bibitem{Gavril} L.N.~Gavrilov, On the geometry of
Gorjatchev-Tchaplygin top, {\it C.R.~Acad. Bulg. Sci.}, {\bf 40}
(1987), 33-36.

\bibitem{FomSG} A.T.~Fomenko, {\it Symplectic Geometry. Methods
and Applications}, Gordon and Breach (1988).

\bibitem{BobReySem} A.I.~Bobenko, A.G.~Reyman,
M.A.~Semenov-Tian-Shansky, The Kowalewski top 99 years later: a Lax
pair, generalizations and explicit solutions, {\it Commun. Math.
Phys.}, {\bf 122}, 2 (1989), 321-354.

\bibitem{ZotRCD} D.B.~Zotev, Fomenko-Zieschang invariant in the
Bogoyavlenskyi case, {\it Regular and Chaotic Dynamics}, {\bf 5}, 4
(2000), 437-458.

\bibitem{Kh32} M.P.~Kharlamov, One class of solutions with two
invariant relations in the problem of motion of the Kowalevsky top
in double constant field, {\it Mekh. tverd. tela}, 32 (2002), 32-38.
(In Russian)

\bibitem{Kh34}
M.P.~Kharlamov, Critical set and bifurcation diagram in the problem
of motion of the Kowalevsky top in double field, {\it Mekh. tverd.
tela}, 34 (2004), 47-58. (In Russian)

\bibitem{KhRCD1} M.P.~Kharlamov, Bifurcation diagrams of the
Kowalevski top in two constant fields, {\it Regular and Chaotic
Dynamics}, {\bf 10}, 4 (2005), 381-398.

\bibitem{KhSav} M.P.~Kharlamov, A.Y.~Savushkin, Separation of
variables and integral manifolds in one partial problem of motion of
the generalized Kowalevski top, {\it Ukr. Math. Bull.}, {\bf 1}, 4
(2004), 548-565.

\bibitem{KhRCD2} M.P.~Kharlamov, Separation of variables in the
generalized 4th Appelrot class, {\it Regular and Chaotic Dynamics},
{\bf 12}, 3 (2007), 267-280.

\bibitem{KhZot} M.P.~Kharlamov, D.B.~Zotev, Non-degenerate energy
surfaces of rigid body in two constant fields, {\it Regular and
Chaotic Dynamics}, {\bf 10}, 1 (2005), 15-19.

\bibitem{Kh361} M.P.~Kharlamov, Special periodic motions of the
generalized Delone case, {\it Mekh. tverd. tela}, 36 (2006), 23-33.
(In Russian)

\bibitem{Kh362} M.P.~Kharlamov, Regions of existence of motions of
the generalized Kovalevskaya top and bifurcation diagrams, {\it
Mekh. tverd. tela}, 36 (2006), 13-22. (In Russian)

\bibitem{BolsFom} A.V.~Bolsinov, A.T.~Fomenko, {\it Integrable
Hamiltonian systems: geometry, topology, classification}, Chapman \&
Hall/CRC (2004).

\bibitem{Zhuk} N.E.~Zhoukovsky, On the motion of a rigid body with
holes filled with a homogeneous fluid, In: {\it Collected Works},
Gostekhizdat, Moscow (1949), {\bf 1}, 31-152. (In Russian)

\bibitem{Appel} G.G.~Appelrot, Non-completely symmetric heavy
gyroscopes. In: {\it Motion of a rigid body about a fixed point},
Moscow-Leningrad (1940), 61-156. (In Russian)

\end{thebibliography}
\end{document}